\documentclass[graybox]{svmult}

\usepackage{amssymb}
\usepackage{amsmath}
\usepackage{latexsym}

\usepackage{mathptmx}       
\usepackage{helvet}         
\usepackage{courier}        
\usepackage{type1cm}        
\usepackage{makeidx}         
\usepackage{graphicx}        
\usepackage{multicol}        
\usepackage[bottom]{footmisc}

\makeindex             

\begin{document}

\title*{On the structure of $\mathbb{N}$-graded Vertex Operator Algebras}
\author{Geoffrey Mason\thanks{Supported by the NSF} and Gaywalee Yamskulna\thanks{Partially supported by a grant from the Simons Foundation (\# 207862)}}
\institute{Department of Mathematics, UC Santa Cruz, California 95064, \email{gem@ucsc.edu}\\
Department of Mathematical Sciences, Illinois State University, Normal, Illinois 61790, \email{gyamsku@ilstu.edu}}
%
%
\maketitle

\abstract*{We consider the algebraic structure of $\mathbb{N}$-graded vertex operator algebras with
conformal grading $V=\oplus_{n\geq 0} V_n$ and $\dim V_0\geq 1$.\ We prove several results along the lines that the vertex operators $Y(a, z)$ for $a$ in a Levi factor
of the Leibniz algebra $V_1$ generate an affine Kac-Moody subVOA. If $V$ arises as a shift of
a self-dual VOA of CFT-type, we show that $V_0$ has a `de Rham structure' with many of the properties
of the de Rham cohomology of a complex connected manifold equipped with Poincar\'{e} duality. }

\abstract{We consider the algebraic structure of $\mathbb{N}$-graded vertex operator algebras with
conformal grading $V=\oplus_{n\geq 0} V_n$ and $\dim V_0\geq 1$.\ We prove several results along the lines that the vertex operators $Y(a, z)$ for $a$ in a Levi factor
of the Leibniz algebra $V_1$ generate an affine Kac-Moody subVOA. If $V$  arises as a shift of
a self-dual VOA of CFT-type, we show that $V_0$ has a `de Rham structure' with many of the properties
of the de Rham cohomology of a complex connected manifold equipped with Poincar\'{e} duality. }



\bigskip
\noindent
Contents\\
1.\ Introduction\\
2.\ Leibniz algebras and vertex operator algebras\\
3.\ $\mathbb{N}$-graded vertex operator algebras\\
4.\ The bilinear form $\langle \ , \ \rangle$\\
5.\ The de Rham structure of shifted vertex operator algebras\\
6.\ The bilinear form in the shifted case\\
7.\ The $C_2$-cofinite case\\
8.\ Examples of shifted vertex operator algebras\\
8.1.\ Shifted $\widehat{sl_2}$\\
8.2.\ Shifted lattice theories

\section{Introduction}
The purpose of the present paper is the study of the algebraic structure of
\emph{$\mathbb{N}$-graded vertex operator algebras} (VOAs).\ A VOA
$V=(V, Y, \mathbf{1}, \omega)$ is called $\mathbb{N}$-graded if it has no nonzero states of negative conformal weight, so that its conformal grading takes the form
\begin{eqnarray}\label{gddim}
V = \oplus_{n=0}^{\infty} V_n.
\end{eqnarray}

The VOAs in this class which have been most closely investigated hitherto are
those of \emph{CFT-type}, where one assumes that $V_0 = \mathbb{C}\mathbf{1}$ is spanned by the vacuum vector.\ (It is well-known that a VOA of CFT-type is necessarily $\mathbb{N}$-graded.)\ Our main interest here is in the contrary case, when $\dim V_0\geq 2$.

\medskip
There are several available methods of constructing $\mathbb{N}$-graded  vertex algebras.\
One that particularly motivates the present paper arises from the cohomology of the chiral de Rham complex of a complex manifold $M$, due to
Malikov, Schechtman and Vaintrob (\cite{MS1}, \cite{MS2}, \cite{MSV}).\ In this construction
$V_0$ (which is always a commutative algebra with respect to the $-1^{th}$ operation $ab:=a(-1)b$)  is identified with the de Rham cohomology $H^*(M)$.\ One can also consider algebraic structures defined on $V_0\oplus V_1$ or closely related spaces, variously called \emph{$1$-truncated conformal algebras, vertex $A$-algebroids, and Lie $A$-algebroids} (\cite{GMS}, \cite{Br}, \cite{LY}), and construct
(loc.\ cit.)\ $\mathbb{N}$-graded vertex algebras from a $1$-truncated conformal algebra much as one constructs affine VOAs from a simple Lie algebra.\ A third method
 involves \emph{shifted} VOAs (\cite{DM3}).\ Here, beginning with a VOA 
$V= (V, Y, \mathbf{1}, \omega)$, one replaces $\omega$ by a second conformal vector
$\omega_h:= \omega +h(-2)\mathbf{1}\ (h \in V_1)$ so that
$V^h:= (V, Y, \mathbf{1}, \omega_h)$ is a new VOA with the same Fock space, vacuum and set of fields
as $V$.\ We call $V^h$ a \emph{shifted VOA}.\  For propitious choices of $V$ and $h$ (lattice theories were used in \cite{DM3}) one can construct lots of shifted VOAs that are $\mathbb{N}$-graded.\ In particular, if $V$ is rational then
$V^h$ is necessarily also rational, and in this way one obtains $\mathbb{N}$-graded rational VOAs that are not of CFT-type.

\medskip
 Beyond these construction techniques, the literature devoted to the study of $\mathbb{N}$-graded VOAs \emph{per se} is sparse.\ There are good reasons for this.\ For a VOA of CFT-type the weight $1$ space $V_1$ carries the structure of a Lie algebra $L$ with respect to the bracket $[ab]=a(0)b\ (a, b\in V_1)$, and the modes of the corresponding vertex operators $Y(a, z)$ close on an affinization $\widehat{L}$ of $L$.\ For a general VOA, $\mathbb{N}$-graded or not, this no longer pertains.\ Rather, $V_1$ satisfies the weaker property of being a \emph{left Leibniz algebra} (a sort of Lie algebra for which skew-symmetry fails), but one can still ask the question:  
 \begin{eqnarray}\label{affalg?}
 && \mbox{what is the nature of the algebra spanned} \notag\\
&& \mbox{by the vertex operators
 $Y(a, z)$ for $a\in V_1$?}
 \end{eqnarray}

 \medskip
 Next we give an overview of the contents of the present paper.\ Section 2 is concerned with  question (\ref{affalg?}) for an \emph{arbitrary} VOA.\ After reviewing general facts about Leibniz algebras and their relation to VOAs, we consider the 
 annihilator $F \subseteq V_1$ of the  \emph{Leibniz kernel} of $V_1$.\ $F$ is itself a Leibniz algebra, and we show (Theorem \ref{thmLevi}) that the vertex operators $Y(a, z)$ for $a$ belonging to a fixed Levi subalgebra $S\subseteq F$ 
close on an affine algebra $U\subseteq V$.\ Moreover, all such Levi factors $F$ are conjugate in Aut($V$), so that $U$ is an invariant of $V$.\ (Finite-dimensional Leibniz algebras have a Levi decomposition in the style of Lie algebras, and the semisimple part is a true Lie algebra.)\ This result generalizes the `classical' case of VOAs of CFT-type discussed above, to which 
it reduces if $\dim V_0=1$, and provides a partial answer to (\ref{affalg?}).\ \ We do not know if, more generally, the same result holds if we replace
$S$ by a Levi factor of  $V_1$.

\medskip
From Section $3$ on we consider simple $\mathbb{N}$-graded VOAs that are also \emph{self-dual} in the sense that they admit a nonzero invariant bilinear form $( \ , \ ): V\times V \rightarrow \mathbb{C}$ (cf.\ \cite{L}).\ By results in \cite{DM2} this implies that
$V_0$ carries the structure of a \emph{local, commutative, symmetric algebra}, and in particular it has a unique minimal ideal $\mathbb{C}t$.\ This result is fundamental for everything that follows.\ It permits
us to introduce a second bilinear form $\langle \ , \rangle: V_1\times V_1 \rightarrow \mathbf{C}$ on $V_1$, defined in terms of $( \ , \ )$ and $t$, and we try to determine its radical.\ Section $3$ covers background results, and in Section $4$ we show (Proposition \ref{relari}) exactly how Rad$\langle \ , \ \rangle$ is related to
the annihilator of the endomorphism $t(-1)$ acting on $V_1$.\ In all cases known to us we have
\begin{eqnarray}\label{anndet}
Rad\langle \ , \ \rangle = Ann_{V_1}(t(-1)),
\end{eqnarray}
and it is of interest to know if this is always true.

\medskip
In Sections $5$ and $6$ we consider shifted VOAs, more precisely we consider the set-up in which we have a \emph{self-dual} VOA $(W, Y, \mathbf{1}, \omega')$ of CFT-type together with an element $h\in W_1$ such that the shifted theory $W^h=(W, Y, \mathbf{1}, \omega'_h)$ as previously defined is a self-dual, $\mathbb{N}$-graded VOA $V$.\ As we mentioned, triples $(W, h, V)$ of this type are readily constructed,  and they have very interesting properties.\ The main result of Section $5$ is Theorem \ref{derham}, which, roughly speaking, asserts that $V_0$ looks just like the de Rham cohomology of a complex manifold equipped with Poincar\'{e} duality.\ More precisely, we prove that the eigenvalues of $h(0)$ acting on $V_0$ are
nonnegative integers; the maximal eigenvalue is $\nu$, say, and the $\nu$-eigenspace is $1$-dimensional and spanned by $t$; and the restriction of the nonzero invariant bilinear form on
$V$ to $V_0$ induces a perfect pairing between the $\lambda$- and $(\nu-\lambda)$-eigenspaces.\
One may compare this result with the constructions of Malikov \emph{et al} in the chiral de Rham complex, where the same conclusions arise directly from the identification of the lowest weight space with $H^*(M)$ for a complex manifold $M$.\ There is, of course, no \emph{a priori} complex manifold
associated to the shifted triple $(W, h, V)$, but one can ask whether, at least in some instances, the
cohomology of the chiral de Rham complex arises from the shifted construction? 

\medskip
In Section $8$ we present several examples that illustrate the theory described in the previous paragraph.\ In particular, we take for $W$ the affine Kac-Moody theory 
$L_{\widehat{sl_2}}(k, 0)$ of positive integral level $k$ and show that it has a canonical shift to a self-dual, $\mathbb{N}$-graded VOA $V=W^H$ ($2H$ is semisimple and part of a Chevalley basis for $sl_2$).\ It turns out that the algebra structure on $V_0$ is naturally identified with 
$H^*(\mathbb{C}\mathbb{P}^k)$.\ We also look at shifts of lattice theories $W_L$, where the precise structure of $V_0$ depends on $L$.

\medskip
Keeping the notation of the previous paragraph, in Section $6$ we use the results of Section $5$ to prove that the shifted VOA $V$ indeed satisfies (\ref{anndet}).\ Moreover, if the Lie algebra
$W_1$ on the weight $1$ space of the CFT-type VOA $W$ is \emph{reductive}, we prove that
$Rad\langle \ , \ \rangle$ is the \emph{nilpotent radical} of the Leibniz algebra $V_1$, i.e., the smallest
ideal in $V_1$ such that the quotient is a reductive Lie algebra.\ It was precisely for the purpose of
proving such a result that the form $\langle \ , \ \rangle$ was introduced.\ It is known   \cite{DM1} that $W_1$ is indeed reductive if $W$ is \emph{regular} (rational and $C_2$-cofinite), so for VOAs obtained as a shift of such a $W$, we get a precise description of the nilpotent radical, generalizing the
corrresponding result of \cite{DM3}.

\medskip
In Section 7 we study simple, self-dual $\mathbb{N}$-graded VOAs that are $C_2$-cofinite.\
We prove (Theorem \ref{thmC2}) that in this case $Rad\langle \ , \ \rangle$ lie between the nilpotent radical of $V_1$ and the solvable radical of $V_1$.\ In particular, the restriction of
$\langle \ , \ \rangle$ to a Levi factor $S\subseteq V_1$ is nondegenerate; furthermore, the vertex operators $Y(a, z)\ (a\in S)$ close on a tensor product of WZW models, i.e. simple affine algebras
$L_{\hat{\frak{g}}}(k, 0)$ of positive integral level $k$.\ Thus we obtain a partial answer to (\ref{affalg?}) which extends results in \cite{DM4}, where the result was proved for CFT-type VOAs.

\section{Leibniz algebras and vertex operator algebras}\label{SS2}
In this Section, we assume that $V$ is any simple vertex operator algebra
\begin{eqnarray*}
V = \oplus_{n \geq n_0} V_n,
\end{eqnarray*}
with \emph{no} restriction on the nature of conformal grading.

\medskip
A \emph{left Leibniz algebra} is a $\mathbb{C}$-linear space $L$ equipped with a bilinear product,
or bracket, $[\  \ ]$ satisfying
\begin{eqnarray*}
[a[bc]] = [[ab]c] + [b[ac]], \ \ (a, b, c \in V).
\end{eqnarray*}
Thus $[a*]$ is a left derivation of the algebra $L$, and $L$ is a Lie algebra if, and only if,
the bracket is also skew-symmetric.
We refer to  \cite{MY} for  facts about Leibniz algebras that we use below.
\begin{lemma}  $V$ is a $\mathbb{Z}$-graded left Leibniz algebra
with respect to the $0^{th}$ operation $[ab]:=a(0)b$.\ Indeed, there is a triangular decomposition
\begin{eqnarray}\label{triangular}
V = \left\{\oplus_{n \leq 0} V_n \right\} \oplus V_1 \oplus \left\{\oplus_{n \geq 2} V_n\right\}
\end{eqnarray}
into left Leibniz subalgebras.\ Moreover, $\oplus_{n \leq 0} V_n$ is nilpotent.
\end{lemma}

\noindent
 {\bf Proof}.\ Recall the commutator formula
\begin{eqnarray}\label{commform}
[u(p), v(q)]w &=& \sum_{i=0}^{\infty} {p \choose i}(u(i)v)(p+q-i)w.
\end{eqnarray}
Upon taking $p=q=0$, (\ref{commform}) specializes to
\begin{eqnarray*}
u(0)v(0)w - v(0)u(0)w = (u(0)v)(0)w,
\end{eqnarray*}
which is the  identity required to make $V$ a left Leibniz algebra.\ The remaining assertions are consequences of
\begin{eqnarray*}
u(0)(V_n) \subseteq V_{n+k-1}\ \ (u \in V_k).
\end{eqnarray*}
$\hfill \Box$
\begin{remark} A \emph{right} Leibniz algebra $L$ has a bracket with respect to which
$L$ acts as right derivations.\ Generally, a left Leibniz algebra is \emph{not} a right Leibniz algebra,
and in particular a vertex operator algebra is generally not a right Leibniz algebra.
\end{remark}

 It is known (e.g.,\ \cite{MY}) that a finite-dimensional left Leibniz algebra has a
 \emph{Levi decomposition}.\ In particular, this applies to the middle summand $V_1$ 
 in (\ref{triangular}).\ Thus there is a
 decomposition 
 \begin{eqnarray}\label{Levidecomp}
 V_1 = S \oplus B
 \end{eqnarray}
  where $S$ is a semisimple \emph{Lie} subalgebra
 and $B$ is the solvable radical of $V_1$.\ As in the case of a Lie algebra, we call
 $S$ a \emph{Levi subalgebra}.\ Unlike Lie algebras, Levi factors
 are generally \emph{not} conjugate to each other by exponential automorphisms, i.e.,
 Malcev's theorem does \emph{not} extend to Leibniz algebras (loc.\ cit.).
 
 \medskip
  This circumstance leads to
 several interesting questions in VOA theory.\ In particular, what is the nature of the subalgebra of $V$ generated by a Levi subalgebra $S\subseteq V_1$?\ Essentially, we want a description of
 the Lie algebra of operators generated by the modes $a(n)\ (a\in S, n\in \mathbb{Z})$.\ In the case when
 $V$ is of CFT-type (i.e., $V_0=\mathbb{C}\mathbf{1}$ is spanned by the vacuum), it is a fundamental fact that these modes generate an affine algebra.\ Moreover, all Levi subfactors of
 $V_1$ are conjugate in Aut$(V)$ (cf.\ \cite{M}), so that the affine algebra is an invariant of $V$.\ It would be
 interesting to know if these facts continue to hold for arbitrary vertex operator algebras.\ We shall deal here with a special case.
 
 \medskip
 To describe our result, introduce the \emph{Leibniz kernel} defined by
 \begin{eqnarray*}
N:= \langle a(0)a\ | \ a\in V_1\rangle=\langle a(0)b+b(0)a\ | \ a, b\in V_1\rangle \ \ (\mbox{linear span}).
\end{eqnarray*}
$N$ is the smallest $2$-sided ideal of $V_1$ such that  $V_1/N$ is a Lie algebra.\ The annihilator of the Leibniz kernel is
\begin{eqnarray*}
F:= Ann_{V_1}(N) = \{a\in V_1 \ | a(0)N=0\}.
\end{eqnarray*}
This is a $2$-sided ideal of $V_1$, in particular it is a Leibniz subalgebra and itself contains Levi factors.\ We will prove
 
\begin{theorem}\label{thmLevi} Let $V$ be a simple vertex operator algebra, with $N$ and
$F$ as above. Then the following hold: \\
(a)\ Aut$(V)$ acts transitively on the Levi subalgebras of $F$.\\
(b)\ Let $S\subseteq F$ be a Levi subalgebra of $F$.\  Then $u(1)v\in \mathbb{C}\mathbf{1}\ (u, v\in S)$,
and the vertex operators $Y(u, z)\ (u\in S)$
close on an affine algebra, i.e.,
\begin{eqnarray*}
[u(m), v(n)]=(u(0)v)(m+n)+m\alpha(u, v)\delta_{m+n, 0}Id_V,
\end{eqnarray*}
where $u(1)v=\alpha(u, v)\mathbf{1}$.
\end{theorem}

We prove the Theorem in a sequence of Lemmas.\ Fix a Levi subalgebra $S\subseteq F$, and set
\begin{eqnarray*}
W:= \oplus_{n \leq 0}V_n.
\end{eqnarray*}

\begin{lemma}\label{lemmaWtriv} $W$ is a \emph{trivial} left $S$-module, i.e., 
$u(0)w = 0 \ (u \in S, w \in W)$.
\end{lemma}
\noindent
{\bf Proof.} We have to show that each homogeneous space $V_n\ (n \leq 0)$, is a trivial
left $S$-module. Because $L(-1):V _n \rightarrow V_{n+1}$ is an \emph{injective}
$V_1$-equivariant map for $n\not= 0$, it suffices to show that $V_0$ is a trivial $S$-module.

\medskip
Consider $L(-1): V_0 \rightarrow V_1$, and set $N':= L(-1)V_0$. Because $[L(-1), u(0)]=0$,
$L(-1)$ is 
$V_1$-equivariant.\ By skew-symmetry we have
$u(0)u = 1/2L(-1)u(1)u$. This shows that $N\subseteq N'$.
Now because $(L(-1)v)(0)=0\ (v \in V)$ then in particular $N'(0)V_1=0$.\ Therefore, 
$S(0)N' = \langle u(0)v+v(0)u\ | \ u \in S, v \in N'\rangle \subseteq N$. But
$S$ is semisimple  and it annihilates $N$. It follows
that $S$ annihilates $N'$.

\medskip
Because $V$ is simple, its center $Z(V)= kerL(-1)$ coincides with $\mathbb{C}\mathbf{1}$.\ By Weyl's theorem of complete reducibility,
there is an $S$-invariant decomposition
\begin{eqnarray*}
V_0 = \mathbb{C}\mathbf{1}\oplus J,
\end{eqnarray*}
and restriction  of $L(-1)$ is an $S$-isomorphism $ J \stackrel{\cong}{\rightarrow} N'$.\ Because $S$ annihilates $N'$,  it must annihilates $J$. It therefore also annihilates $V_0$, as we see from the previous display. 
This completes the proof of the Lemma. $\hfill \Box$

\medskip
\begin{lemma}\label{nilaction} We have
\begin{eqnarray}
u(k)w=0\ \ (u \in S, w \in W, k \geq 0).
\end{eqnarray}
\end{lemma}
\noindent
{\bf Proof}. Because $S$ is semisimple, we may, and shall, assume without loss that  $u$ is a commutator $u= a(0)b\ (a, b \in S)$. Then
\begin{eqnarray*}
(a(0)b)(k)w = a(0)b(k)w - b(k)a(0)w= 0.
\end{eqnarray*}
The last equality holds thanks to Lemma \ref{lemmaWtriv}, and because
$b(k)w \in W$ for $k \geq 0$. The Lemma is proved. $\hfill \Box$

\begin{lemma}\label{lemmacommform} We have
\begin{eqnarray}\label{commform2}
[u(m), w(n)]=0\ \ (u \in S, w \in W; m, n \in \mathbb{Z}).
\end{eqnarray}
\end{lemma}
\noindent
{\bf Proof.}  First notice that by Lemma \ref{lemmaWtriv},
\begin{eqnarray}\label{commform3}
[u(0), w(n)] = (u(0)w)(n)=0.
\end{eqnarray}
Once again, it is suffices to assume that $u=a(0)b\ (a, b \in F)$.
 In this case we obtain, using several applications of (\ref{commform3}),
\begin{eqnarray*}
[u(m), w(n)]  &=& [(a(0)b)(m), w(n)] \\
&=& [[a(0), b(m)], w(n)] \\
&=& [a(0), [b(m), w(n)]] -[b(m), [a(0), w(n)]]\\
&=& [a(0), (b(0)w)(m+n)+m(b(1)w)(m+n-1)] \\
&=&  0.
\end{eqnarray*}
This completes the proof of the Lemma. $\hfill \Box$

\medskip
Consider the Lie algebra $L$ of operators on $V$ defined by
\begin{eqnarray*}
L := \langle u(m), w(n) \ | \ u \in S, w \in W; m, n \in \mathbb{Z} \rangle.
\end{eqnarray*}

\medskip
If $w, x \in W$ then
\begin{eqnarray*}
[w(m), x(n)] = \sum_{i \geq 0} {m \choose i}(w(i)x)(m+n-i),
\end{eqnarray*}
and $w(i)x$ has weight \emph{less} than that of $w$ and $x$ whenever $w, x \in W$ are homogeneous
and $i \geq 0$. This shows that the operators $w(m)\ (w \in W, m \in \mathbb{Z})$
span a nilpotent ideal of $L$, call it $P$. Let $L_0$ be the Lie subalgebra generated by $u(m)\ (u \in S_0, m \in \mathbb{Z})$.
 By Lemma \ref{lemmacommform}, $L_0$ is also an ideal of $L$, indeed 
  \begin{eqnarray*}
L = P + L_0, \ [P, L_0]=0.
\end{eqnarray*}

Next, for $u, v \in S$ we have
\begin{eqnarray}\label{S0calc}
[u(m), v(n)] &=& (u(0)v)(m+n)+\sum_{i \geq 1}{m \choose i}(u(i)v)(m+n-i).
\end{eqnarray}
So if $w \in S$ then by Lemma \ref{lemmacommform} once more,
\begin{eqnarray*}
[w(0), [u(m), v(n)] ] = [w(0), (u(0)v)(m+n)] = (w(0)(u(0)v))(m+n).
\end{eqnarray*}
This shows that $L_0$ coincides with its derived subalgebra.
Furthermore, the short exact sequence
\begin{eqnarray*}
0 \rightarrow P\cap L_0 \rightarrow L_0 \rightarrow L_0/(P\cap L_0) \rightarrow 0
\end{eqnarray*}
shows that $L_0$ is a perfect 
central extension of the loop algebra $\widehat{L}(S_0) \cong$ \\
$ L_0/(P\cap L_0)$. Because $H^2(\widehat{L}(G))$ is $1$-dimensional
for a finite-dimensional simple Lie algebra $G$, we can conclude that $\dim (P\cap L_0)$ is
\emph{finite}. 

\medskip
Taking $m=1$ in (\ref{S0calc}), it follows that
\begin{eqnarray}\label{u1v}
(u(1)v)(n) \in P\cap L_0 \ (n \in \mathbb{Z}).
\end{eqnarray}
Now if $u(1)v \notin Z(V)$ then all of the modes $(u(1)v)(n), n < 0,$ are nonzero, and indeed linearly
independent. This follows from the creation formula \\
 $\sum_{n\leq -1}(u(1)v)(n)\mathbf{1}z^{-n-1} = e^{zL(-1)}u(1)v$. Because $P\cap L_0$ is finite-dimensional and contains all of these modes, this is not possible. We deduce that in fact
 $u(1)v \in Z(V)=\mathbb{C}$, say $u(1)v = \alpha(u, v)\mathbf{1}, \alpha(u, v)\in \mathbb{C}$.
 
 \medskip
 Taking $m=2, 3, ...$ in (\ref{S0calc}), we argue in the same way that $u(i)v \in Z(V)$
 for $i\geq 2$. Since $Z(V)\subseteq V_0$, this means that $u(i)v=0$ for $i \geq 2$.
 Therefore, (\ref{S0calc}) now reads
 \begin{eqnarray}\label{S1calc}
[u(m), v(n)] &=& (u(0)v)(m+n)+m\alpha(u, v)\delta_{m+n, 0}Id,
\end{eqnarray}
where $u(1)v = \alpha(u, v)\mathbf{1}$.\ This completes the proof of part (b) of the Theorem.

\medskip
It remains to show that Aut$(V)$ acts transitively on the set of Levi subalgebras
of $F$. 
\begin{lemma}\label{lemprim} $[FF]$ consists of \emph{primary states}, i.e., $L(k)[FF]=0\ (k \geq 1)$.
\end{lemma}
\noindent
{\bf Proof}. It suffices to show that $L(k)a(0)b=0$ for $a, b \in F$ and $k \geq 1$. Since
$L(k)b \in W$ then $a(0)L(k)b=0$ by Lemma \ref{lemmaWtriv}. Using induction on $k$, we then have
\begin{eqnarray*}
L(k)a(0)b &=&[L(k), a(0)]b \\
&=&(L(-1)a)(k+1)b +(k+1)(L(0)a)(k)b + (L(k)a)(0)b \\
&=&(L(k)a)(0)b \\
&=& \sum_{i \geq 0}(-1)^{i+1}/i! L(-1)^i  b(i)L(k)a = 0,
\end{eqnarray*}
where we used skew-symmetry for the fourth equality, and
Lemma \ref{nilaction} for the last equality. The Lemma is proved. $\hfill \Box$

\medskip
Finally,  by \cite{MY}, Theorem 3.1, if $S_1, S_2$ are a pair of Levi subalgebras
of $F$, then we can find $x \in [FF]$ such that $e^{x(0)}(S_1)=S_2$. Because $x$ is a primary state, it is well-known that $e^{x(0)}$ is an automorphism of $V$.\ This completes the proof of Theorem
\ref{thmLevi}. $\hfill \Box$

\section{$\mathbb{N}$-graded vertex operator algebras}
In this Section, we assume that $V$ is a \emph{simple, self-dual, $\mathbb{N}$-graded} vertex operator
algebra.\
We are mainly interested in the case that $\dim V_0\geq 2$.\ There is a lot of structure available to us in this situation, and in this Section we review some of the details, and at the same time introduce some salient notation.

\medskip
The self-duality of $V$ means that there is a \emph{nonzero}  bilinear form
\begin{eqnarray*}
( \ , \ ): V \times V \rightarrow \mathbb{C}
\end{eqnarray*}
that is \emph{invariant} in the sense that 
\begin{eqnarray}\label{invbilform1}
\left(Y(u, z)v, w\right)= \left(v, Y(e^{zL(1)}(-z^{-2})^{L(0)}u, z^{-1})w\right)\ \ (u, v, w\in V).
\end{eqnarray}
 $( \ , \ )$ is necessarily \emph{symmetric} (\cite{FHL}), and because $V$ is simple
then it is \emph{nondegenerate}.\ The simplicity of $V$ also implies
(Schur's Lemma) that $(\ , \ )$ is \emph{unique} up to scalars.\ By results of Li \cite{L}, there is an isomorphism between the space of invariant bilinear forms and $V_0/L(1)V_1$.\
Therefore, $L(1)V_1$ has codimension $1$ in $V_0$.\  For now, we fix a
nonzero form $(\ , \ )$, but do not choose any particular normalization.

\medskip
 If $u\in V_k$ is \emph{quasiprimary} (i.e., $L(1)u=0$), then
(\ref{invbilform1}) is equivalent to
\begin{eqnarray}\label{invbilform2}
(u(n)v, w) = (-1)^k(v, u(2k-n-2)w)\ \ \ (n \in \mathbb{Z}).
\end{eqnarray}
In particular, taking $u$ to be the conformal vector
$\omega \in V_2$, which is always quasiprimary, and $n=1$ or $2$ yields
\begin{eqnarray}
(L(0)v, w) &=& (v, L(0)w), \label{invbilform3} \\
(L(1)v, w) &=& (v, L(-1)w). \label{invbilform30}
\end{eqnarray}

\medskip
 We write $P\perp Q$ for 
the direct sum of subspaces  $P, Q \subseteq V$ that are orthogonal
with respect to $( \ , \ )$. Thus $(V_n, V_m)=0$ for $n\not= m$ by (\ref{invbilform3}), so that
\begin{eqnarray*}
V = \perp_{n\geq 0} V_n.
\end{eqnarray*}
In particular, the restriction of $( \ , \ )$ to each $V_n$ is nondegenerate.
We adopt the following notational convention for
$U \subseteq V_n$:
\begin{eqnarray*}
U^{\perp} := \{ a \in V_n \ | \ (a, U)=0\}.
\end{eqnarray*}

The \emph{center} of $V$ is defined to be $Z(V):=$ ker$L(-1)$.\ Because $V$ is simple, 
we have $Z(V) = \mathbb{C}\mathbf{1}$ (cf.\ \cite{LL}, \cite{DM2}).\ Then from (\ref{invbilform30}) we find that
\begin{eqnarray}\label{L1V1}
(L(1)V_1)^{\perp} &=& \mathbb{C}\mathbf{1}.
\end{eqnarray}

$V_0$ carries the structure of a \emph{commutative associative algebra} with respect to the operation $a(-1)b\ (a, b \in V_0)$.\ 
Since all elements
in $V_0$ are quasiprimary, we can apply (\ref{invbilform2}) with $u, v, w \in V_0$ to obtain
\begin{eqnarray}\label{invbilform21}
(u(-1)v, w) = (v, u(-1)w).
\end{eqnarray}
Thus $( \ ,\ )$ is a nondegenerate, symmetric, invariant bilinear form on $V_0$, 
whence $V_0$ is a commutative \emph{symmetric algebra}, or  \emph{Frobenius algebra}. 

\medskip
What is particularly important for us is that because 
$V$ is simple, $V_0$ is  a \emph{local algebra}, i.e.,  the Jacobson radical $J:= J(V_0)$ is the unique maximal
ideal of $V_0$, and every element of $V_0\setminus{J}$ is a unit.\ This follows from
results of Dong-Mason (\cite{DM2}, Theorem 2 and Remark 3).

\medskip
For a symmetric algebra, the map $I \rightarrow I^{\perp}$
is an inclusion-reversing duality on the set of ideals.\
In particular, because
$V_0$ is a local algebra, it has a \emph{unique minimal} (nonzero) ideal, call it $T$, and
$T$ is $1$-dimensional. Indeed,
\begin{eqnarray}\label{Tchar}
T = J^{\perp} = Ann_{V_0}(J) = \mathbb{C}t,
\end{eqnarray}
for some  fixed, but arbitrary, nonzero element
$t \in T$.\ We have 
\begin{eqnarray*}
 T \oplus L(1)V_1= V_0.
\end{eqnarray*}
This is a consequence of
the nondegeneracy of $(\ , \ )$ on $V_0$, which entails that
$L(1)V_1$ contains no nonzero ideals of $V_0$. In
particular,  (\ref{L1V1}) implies that
\begin{eqnarray}\label{1tperp}
(t, \mathbf{1})\not= 0.
\end{eqnarray}

\medskip
We will change some of the notation from the previous Section by here setting
$N:=L(-1)V_0$ (it was denoted $N'$ before).\ In the proof of Lemma
\ref{lemmaWtriv} we showed that $N$ contains the Leibniz kernel of $V_1$.\ In particular, $V_1/N$ is a Lie algebra.
We write
\begin{eqnarray}\label{Leibnilraddef}
N_0/N = Nil(V_1/N),\ N_1/N = Nilp(V_1/N), B/N=solv(V_1/N),
\end{eqnarray}
the \emph{nil radical}, \emph{nilpotent radical}, and \emph{solvable radical} respectively of $V_1/N$.\ $N_0/N$ is the largest
nilpotent ideal in $V_1/N$, $N_1/N$ is the intersection of the
annihilators of simple $V_1/N$-modules, and $B/N$ the largest solvable ideal in $V_1/N$.\ It is 
well-known that $N_1\subseteq N_0\subseteq B$.\ Moreover, $N_1/N = [V_1/N, V_1/N]\cap B/N$,
$V_1/N_1$ is
a \emph{reductive} Lie algebra, and $N_1$ is the smallest ideal in $V_1$
with this property.\ Note that  $N_0$ and $B$ are also the largest nilpotent, and solvable ideals respectively in the left Leibniz algebra $V_1$.

\medskip
Each of the homogeneous spaces $V_n$ is a left $V_1$-module with respect to the $0^{th}$
bracket.\ Because $u(0)=0$ for $u\in N$, it follows that $V_n$ is also a left module over the Lie algebra
$V_1/N$.\
Since $V_0=\mathbb{C}\mathbf{1}
\oplus J$,  $L(-1)$ induces an \emph{isomorphism} of $V_1$-modules
\begin{eqnarray}\label{LJiso}
L(-1): J \stackrel{\cong}{\rightarrow} N.
\end{eqnarray}

\begin{remark}
 Most of the structure we have been discussing concerns the $1$-truncated conformal algebra
 $V_0\oplus V_1$ (\cite{Br}, \cite{GMS}, \cite{LY}), and many of our results can be couched in this
 language.
\end{remark}

\section{The bilinear form $\langle \ , \ \rangle$}\label{SS3.2}
We keep the notation of the previous Section, in particular $t\in V_0$ spans the unique minimal ideal
of $V_0$.\ We  introduce the bilinear form $\langle \ , \ \rangle: V_1 \otimes V_1 \rightarrow  \mathbb{C}$,
defined as follows:
\begin{equation} \label{<>def}
\langle u, v \rangle:=  (u(1)v, t), \ \ (u, v \in V_1).
\end{equation} 
We are interested in the \emph{radical} of $\langle \ , \ \rangle$, defined as
\begin{eqnarray*}
rad\langle \ , \rangle := \{u \in V_1 \ | \ \langle u, V_1 \rangle = 0\}.
\end{eqnarray*}
We will see that $\langle \ , \ \rangle$ is a symmetric, invariant bilinear form on the Leibniz algebra $V_1$.\ The main result of this Section (Proposition \ref{relari}) determines the radical in terms of certain other subspaces
that we introduce in due course.\
In order to study $\langle \ , \ \rangle$ and its radical, we need some preliminary results.
\begin{lemma}\label{lemmaM} We have $u(0)J \subseteq J$ and $u(0)T\subseteq T$ for $u \in V_1$. Moreover, the left annihilator
\begin{eqnarray*}
M:= \{u \in V_1 \ | \ u(0)T=0\}
\end{eqnarray*}
is a $2$-sided ideal of $V_1$ of codimension $1$, and $M=(L(-1)T)^{\perp}$.
\end{lemma}
\noindent
{\bf Proof.}\ Any derivation of a finite-dimensional commutative algebra
$B$, say, leaves invariant both the Jacobson radical $J(B)$ and its
annihilator. In the case that the derivation
is $u(0), u \in V_1,$ acting on $V_0$, this says that the left action of $u(0)$ leaves
both $J$ and $T$ invariant (using (\ref{Tchar}) for the second assertion).
This proves the first two statements of the Lemma.

\medskip
For $u \in V_1$ we have
\begin{eqnarray}\label{tucalc}
(L(-1)t, u) = (t(-2)\mathbf{1}, u)=(\mathbf{1}, t(0)u) = -(\mathbf{1}, u(0)t).
\end{eqnarray}
Now because $T$ is the unique minimal ideal in $V_0$ then
$T\subseteq J$ and hence $dimL(-1)T=1$ by (\ref{LJiso}).\ Then
(\ref{tucalc}) and (\ref{1tperp}) show that $(L(-1)T)^{\perp}= M$ has codimension exactly $1$
in $V_1$.

\medskip
Finally, using the commutator formula $[u(0), v(0)] = (u(0)v)(0)$ applied
with one of $u, v \in M$ and the other in $V_1$, we see that
$(u(0)v)(0)T=0$ in either case. Thus $u(0)v \in M$, 
whence $M$ is a $2$-sided ideal in $V_1$. This completes the proof of
the Lemma. $\hfill \Box$

\begin{lemma}\label{lemmat2J}We have
\begin{eqnarray*}
t(-2)J=0.
\end{eqnarray*}
\end{lemma}
\noindent
{\bf Proof.} Let $a \in J, u \in V_1$. Then
\begin{eqnarray*}
(t(-2)a, u) = (a, t(0)u) = -(a, u(0)t) = 0.
\end{eqnarray*}
The last equality follows from $u(0)t \in T$ (Lemma \ref{lemmaM})
and $T=J^{\perp}$ (\ref{Tchar}). We deduce that $t(-2)J \subseteq V_1^{\perp}=0$, and
the Lemma follows. $\hfill \Box$

\begin{proposition}\label{propbilform} $\langle \ , \ \rangle$ is a symmetric bilinear form that is \emph{invariant} in the sense that
\begin{eqnarray*}
\langle v(0)u,w \rangle= \langle v, u(0)w \rangle \ \ (u, v, w \in V_1).
\end{eqnarray*}
 Moreover
$N \subseteq$ rad $\langle \ , \ \rangle$. 
\end{proposition}
\noindent
{\bf Proof.} By skew-symmetry we have $u(1)v = v(1)u$ for $u, v \in V_1$, so the symmetry
of $\langle \ , \ \rangle$ follows immediately from the definition (\ref{<>def}).\ If $u\in N$ then $u=L(-1)a$ for some $a\in J$ by
(\ref{LJiso}), and we  have 
\begin{eqnarray*}
\langle u,v \rangle&=&((L(-1)a)(1)v,t) = -(a(0)v,t)\\
&=&-(v,a(-2)t) = (v, t(-2)a-L(-1)t(-1)a) = 0.
\end{eqnarray*}
Here, we used $t(-2)a=0$ (Lemma \ref{lemmat2J}) and $t(-1)a \in t(-1)J=0$
to obtain the last equality. This proves the assertion that $N \subseteq$ rad$\langle \ , \ \rangle$.

\medskip
As for the invariance, we have
\begin{eqnarray*}
\langle u(0)v,w\rangle&=&((u(0)v)(1)w,t) = (u(0)v(1)w-v(1)u(0)w,t). \\
&=&(u(0)v(1)w, t)-\langle v,u(0)w\rangle.
\end{eqnarray*}
Now $V_0= \mathbb{C}\mathbf{1}\oplus J, u(0)\mathbf{1}=0$, and $u(0)J\subseteq J=T^{\perp}$.
Therefore, $(u(0)v(1)w, t)=0$, whence we obtain
$\langle u(0)v,w\rangle =-\langle v,u(0)w\rangle$ from the previous display. Now because
$N\subseteq$ rad$\langle \ , \ \rangle$ we see that
\begin{eqnarray*}
\langle v(0)u, w \rangle &=& -\langle u(0)v-L(-1)u(1)v, w\rangle= \langle v, u(0)w\rangle,
\end{eqnarray*}
as required. This completes the proof of the Proposition. $\hfill \Box$

\medskip
\begin{lemma}\label{lemmarad} We have 
\begin{eqnarray}\label{newinn}
\langle u, v \rangle = -(v, u(-1)t), \ \ u, v \in V_1.
\end{eqnarray}
In particular,
\begin{eqnarray*}
rad\langle \ , \ \rangle = \{u \in V_1 \ | \ u(-1)t=0\}.
\end{eqnarray*}
\end{lemma}
\noindent
{\bf Proof.} The first statement implies the second, so it suffices to establish
(\ref{newinn}). To this end, we apply (\ref{invbilform1}) with $u, v \in V_1, w = t$
to find that
\begin{eqnarray*}
\langle u, v \rangle &=& (u(1)v, t) = -(v, u(-1)t)-(v, (L(1)u)(-2)t).
\end{eqnarray*}
On the other hand, $L(1)u \in V_0 = \mathbb{C}\mathbf{1}\oplus J$, so that
$(L(1)u)(-2)t=a(-2)t = -t(-2)a+L(-1)t(-1)a$ for some $a\in J$. Since
$t(-1)a=t(-2)a=0$ (the latter equality thanks to Lemma \ref{lemmat2J}), the final term of
the previous display vanishes, and what remains is (\ref{newinn}).\ The Lemma is proved.$\hfill \Box$

\bigskip
We introduce
\begin{eqnarray*}
&& \ \ \ \ \ \ \ \ \ \ \ \ \ \ \ \  P:=\{u \in V_1 \ | \ \langle u, M \rangle = 0\},\\
&&Ann_{V_1}(t(-1)) := \{ u\in V_1 \ | \ t(-1)u=0\}.
 \end{eqnarray*}

\begin{lemma}\label{lemmaP} We have
\begin{eqnarray*}
P = \{u \in V_1 \ | \ t(-1)u \in L(-1)T\},
\end{eqnarray*}
and this is a $2$-sided ideal of $V_1$.
\end{lemma}
\noindent
{\bf Proof.} Let $m \in M, u \in V_1$. By (\ref{newinn}) we have
\begin{eqnarray*}
\langle u, m \rangle =-(m, u(-1)t).
\end{eqnarray*}
But by Lemma \ref{lemmaM} we have $M^{\perp}=L(-1)T$.  Hence, the last display implies
that $P = \{u \in V_1 \ | \ u(-1)t \in L(-1)T\}.$
Furthermore, we have $u(-1)t = t(-1)u-L(-1)t(0)u=t(-1)u+L(-1)u(0)t \in t(-1)u+L(-1)T$
by Lemma \ref{lemmaM} once more. Thus $u(-1)t \in L(-1)T$ if, and only if,
$t(-1)u \in L(-1)T$. The first assertion of the Lemma follows.

\medskip
Because $N\subseteq$ rad$\langle \ , \ \rangle$ thanks to Proposition \ref{propbilform},
then certainly $N \subseteq P$. So in order to show
that $P$ is a $2$-sided ideal in $V_1$, it suffices to show that it is a right ideal. To see this, let
$a\in P, n \in M,  u\in V_1$. By Lemma \ref{lemmaM} and Proposition \ref{propbilform}  we find that
\begin{eqnarray*}
\langle a(0)u, n\rangle = \langle a, u(0)n\rangle \in \langle a, M\rangle = 0.
\end{eqnarray*}
This completes the proof of the Lemma. $\hfill \Box$

\begin{lemma}\label{lemmaAR} We have $M\cap Ann_{V_1}(t(-1))=M\cap rad\langle \ , \ \rangle$.
\end{lemma}
\noindent
{\bf Proof.} If $u \in M$ then $t(-1)u = u(-1)t-L(-1)u(0)t = u(-1)t$.
Hence for $u \in M$, we have $u \in Ann_{V_1}(t(-1)) \Leftrightarrow
u(-1)t=0 \Leftrightarrow u \in rad\langle \ , \ \rangle$, where we used
Lemma \ref{lemmarad} for the last equivalence. The Lemma follows.
$\hfill \Box$

\begin{lemma}\label{lemmaRAM} At least one of the containments
$rad\langle \ , \ \rangle \subseteq M,\ Ann_{V_1}(t(-1))\subseteq M$ holds.
\end{lemma}
\noindent
{\bf Proof.} Suppose that we can find $v \in rad\langle \ , \ \rangle \setminus{M}$.
Then $v(-1)t=0$ by Lemma \ref{lemmarad}, and $v(0)t = \lambda t$ for a
scalar $\lambda\not= 0$. Rescaling $v$, we may, and shall, take $\lambda=1$. Then
\begin{eqnarray*}
0 &=& v(-1)t = t(-1)v-L(-1)t(0)v \\
&=& t(-1)v+L(-1)v(0)t = t(-1)v+L(-1)t.
\end{eqnarray*}
Then for $u \in Ann_{V_1}(t(-1))$ we have
\begin{eqnarray*}
(L(-1)t, u)=-(t(-1)v, u) = -(v, t(-1)u)=0,
\end{eqnarray*}
which shows that $Ann_{V_1}(t(-1))\subseteq (L(-1)t)^{\perp} = M$ (using Lemma \ref{lemmaM}). This completes the proof of the
Lemma. $\hfill \Box$

\bigskip
The next result almost pins down
the radical of $\langle \ , \rangle$.

\begin{proposition}\label{relari} Exactly one of the following holds:
\begin{eqnarray*}
&&\ \ (i)\ Ann_{V_1}(t(-1))=rad\langle \ , \ \rangle>\subset P; \\
&&\ (ii)\ Ann_{V_1}(t(-1))\subset rad\langle \ , \ \rangle=P; \\
&&(iii)\ rad\langle \ , \ \rangle\subset Ann_{V_1}(t(-1))=P.
\end{eqnarray*}
In each case, the containment $\subset $ is one in which the smaller subspace
has \emph{codimension one} in the larger subspace.
\end{proposition}
\noindent
{\bf Proof.} First note from Lemma \ref{lemmaP} that $Ann_{V_1}(t(-1))\subseteq P$;
indeed, since $\dim L(-1)T=1$ then the codimension is at most $1$. Also, it is clear from
the definition of $P$ that $rad\langle \ , \ \rangle \subseteq P$.

\medskip
Suppose first that the containment $Ann_{V_1}(t(-1))\subset P$ is  \emph{proper}.\ 
Then we can choose
 $v\in P\setminus{Ann_{V_1}t(-1)}$ such that $t(-1)v=L(-1)t\not= 0$. If
 $u \in Ann_{V_1}(t(-1))$ we then obtain
 \begin{eqnarray*}
(L(-1)t, u) = (t(-1)v, u) = (v, t(-1)u) = 0,
\end{eqnarray*}
whence $u \in (L(-1)t)^{\perp} = M$ by Lemma \ref{lemmaM}. This shows that
$Ann_{V_1}(t(-1))\subseteq M$.\ By Lemma \ref{lemmaAR} it follows that
$Ann_{V_1}(t(-1)) = M\cap rad\langle \ , \ \rangle$.\ Now if also $rad\langle \ , \rangle \subseteq M$
then Case 1 of the Theorem holds. On the other hand, if
$rad\langle \ , \rangle \not\subseteq M$ then we have $Ann_{V_1}(t(-1))\subset rad\langle \ , \rangle
\subseteq P$ and the containment is proper; since $Ann_{V_1}(t(-1))$ has codimension at most
$1$ in $P$ then we are in Case 2 of the Theorem.

\medskip
It remains to consider the case that $Ann_{V_1}(t(-1))=P \supseteq rad\langle \ , \rangle$.
Suppose the latter containment is proper. Because $M$ has codimension $1$ in $V_1$, it follows from
Lemma \ref{lemmaAR} that $rad\langle \ , \ \rangle$ has codimension exactly $1$ in 
$Ann_{V_1}(t(-1))$,  whence Case 3 of the Theorem holds.\ The only remaining possibility is that
$Ann_{V_1}(t(-1))=P = rad\langle \ , \rangle$, and we have to show that this  cannot
occur. By Lemma \ref{lemmaRAM} we must have $rad\langle \ , \ \rangle \subseteq M$, 
so that $M/rad\langle \ , \ \rangle$ is a subspace of codimension $1$ in
the nondegenerate space $V_1/rad\langle \ , \ \rangle$ (with respect to $\langle \ , \ \rangle$).
But then the space orthogonal to $M/rad\langle\ , \ \rangle$, that is $P/rad\langle \ , \ \rangle$,
is $1$-dimensional. This contradiction completes the proof of the Theorem.
$\hfill \Box$

\begin{remark}\label{remQ} In all cases that we know of, it is (i) of Proposition \ref{relari} that  holds.\
This circumstance leads us to raise the question, whether this is always the case?\
We shall later see several rather general situations where this is so.\ At the same time, we
will see how $rad \langle \ , \ \rangle$ is related to the  Leibniz algebra 
structure of $V_1$.
\end{remark}

\section{The de Rham structure of shifted vertex operator algebras}\label{Sshift}
In the next few Sections we consider $\mathbb{N}$-graded vertex operator algebras that are
\emph{shifts} of vertex operator algebras of CFT-type (\cite{DM3}).

\medskip
Let us first recall the idea of a \emph{shifted} vertex operator algebra (\cite{DM3}).
 Suppose that
$W = (W, Y, \mathbf{1}, \omega')$ is an 
$\mathbb{N}$-graded vertex operator algebra of central charge $c'$ and $Y(\omega', z) =: \sum_n L'(n)z^{-n-2}$.\  
It is easy to see that for any $h\in W_1$, the state $\omega_h':= \omega'+L'(-1)h$ is also a Virasoro vector,
i.e., the modes of $\omega'_h$ satisfy the relations of a Virasoro algebra of some central charge
$c'_h$ (generally different from $c'$).\ (The proof of Theorem 3.1 in \cite{DM3} works in the more general result stated here.)\ Now consider the quadruple
\begin{eqnarray}\label{shiftVOA}
W^h := (W, Y, \mathbf{1}, \omega'_h),
\end{eqnarray}
which is generally \emph{not} a vertex operator algebra.\ If it \emph{is}, we call it a \emph{shifted}
vertex operator algebra.

\medskip
We emphasize that in this situation, $W$ and $W^h$ share the \emph{same} underlying Fock space,
the \emph{same} set of vertex operators, and the \emph{same} vacuum vector.\ Only the Virasoro
vectors differ, although this has a dramatic effect because it means that $W$ and $W^h$ have
quite different conformal gradings, so that the two vertex operator algebras seem
quite different.

\medskip
Now let $V=(V, Y, \mathbf{1}, \omega)$ be a simple, self-dual $\mathbb{N}$-graded vertex operator algebra as in the previous two Sections. The assumption of this Section is that there is a self-dual VOA $W$ of CFT-type such that
$W^h=V$.\ That is, $V$ arises as a shift of a vertex operator algebra of CFT-type as described above.\
Thus $h\in W_1$ and
\begin{eqnarray*}
(W, Y, \mathbf{1}, \omega'_h) = (V, Y, \mathbf{1}, \omega).
\end{eqnarray*}
(Note that by definition, $W$ has CFT-type if
$W_0 = \mathbb{C}\mathbf{1}$.\ In this case, $W$ is necessarily $\mathbb{N}$-graded
by \cite{DM3}, Lemma 5.2.) Although the two vertex operator algebras share the same Fock space,
it is convenient to distinguish between them, and we shall do so in what follows.\ We sometimes refer to
$(W, h, V)$ as a \emph{shifted triple}.\ 
Examples are constructed in \cite{DM3}, and it is evident from those calculations that there are large numbers of shifted triples.

\medskip
There are a number of consequences of the circumstance that $(W, h, V)$ is a shifted triple.\
We next discuss some that we will need.\ Because $\omega = \omega_h'=\omega'+L'(-1)h$ then
\begin{eqnarray}\label{Lshift}
L(n) = (\omega'+L(-1)h)(n+1) = L'(n)-(n+1)h(n),
\end{eqnarray}
in particular $L(0)=L'(0)-h(0)$.\
Because $h \in W_1$, we also have $[L'(0), h(0)]=0$.\ Then because $L'(0)$ is semisimple with integral eigenvalues, the same is true of $h(0)$.\ Set
\begin{eqnarray*}
W_{m, n} :=\{w \in W\ | \ L'(0)w=mw, h(0)w=nw\}.
\end{eqnarray*}
Hence,
\begin{eqnarray*}
V_n =  \oplus_{m\geq 0} W_{m, m-n}, 
\end{eqnarray*}
and in particular
\begin{eqnarray}
V_0 &=& \mathbb{C}\mathbf{1}\oplus_{m\geq 1} W_{m, m}, \label{h0decomp0} \\
V_1 &=& \oplus_{m\geq 1} W_{m, m-1}.  \label{h1decomp0}
\end{eqnarray}
(\ref{h0decomp0}) follows because $W$ is of CFT-type, so that
$W_{0, 0}=\mathbb{C}\mathbf{1}$ and $W_{m, n}=0$ for $n<0$.

\medskip
We have $L(0)h = L'(0)h-h(0)h$.\ Because $W$ is of CFT-type then $W_1$ is a Lie algebra with respect
to the $0^{th}$ bracket, and in particular $h(0)h=0$.\ Therefore, $L(0)h=h$, that is
$h\in V_1$.\ Thus $h(0)$ induces a derivation in its action on the commutative algebra $V_0$.\
The decomposition (\ref{h0decomp0}) is one of $h(0)$-eigenspaces, and it confers on $V_0$ 
a structure that looks very much like
the de Rham cohomology of a (connected) complex manifold equipped with its
Poincar\'{e} duality.\ This is what we mean by the \emph{de Rham structure of $V_0$}.\
Specifically, we have

\begin{theorem}\label{derham} Set $A^{\lambda}:= W_{\lambda, \lambda}$, the $\lambda$-eigenspace for
the action of $h(0)$ on $V_0$. Then the following hold;
\begin{eqnarray*}
&&\ \ \ (i)\ \mbox{$A = \oplus_{\lambda} A^{\lambda}$, and if $A^{\lambda}\not= 0$ then $\lambda$ is a \emph{nonnegative} integer.} \\
&&\ \ (ii)\ A^0 = \mathbb{C}\mathbf{1}. \\
&&\ (iii)\ \mbox{Let}\ h(1)h = (\nu/2)\mathbf{1}.\ \mbox{Then}\ A^{\nu}=T=\mathbb{C}t. \\
&&\ \ (iv)\ A^{\lambda}(-1)A^{\mu} \subseteq A^{\lambda+\mu}.\\
&&\ \ \  (v)\  A^{\lambda} \perp A^{\mu}=0\ \mbox{if}\ \lambda+\mu\not= \nu. \\
&&\ \  (vi)\ \mbox{If $\lambda+\mu = \nu$,  the bilinear form $( \ , \ )$ induces a perfect pairing}\\
&&\ \ \ \ \ \ \ \   A^{\lambda} \times A^{\mu} \rightarrow \mathbb{C}. 
\end{eqnarray*}
(Here, $( \ , \ )$ is the invariant bilinear form on $V$, and $T$ the unique
minimal ideal of $V_0$, as in Sections $4$ and $5$.)
\end{theorem}
\noindent
{\bf Proof.} (i) and (ii) are just restatements of the decomposition (\ref{h0decomp0}).

\medskip
Next we prove (iv). Indeed, because $h(0)$ is a derivation of the algebra $A$,
if $a \in A^{\lambda}, b \in A^{\mu}$, then
$h(0)a(-1)b = [h(0), a(-1)]b +a(-1)h(0)b = (h(0)a)(-1)b+\mu a(-1)b= (\lambda+\mu)a(-1)b.$
Part (iv) follows.

\medskip
Next we note that because $W$ is of CFT-type then certainly $h(1)h = (\nu/2)\mathbf{1}$
for some scalar $\nu$. Now let $a, b$ be as in the previous paragraph. Then
\begin{eqnarray*}
\lambda(a, b) &=& (h(0)a, b) = Res_z (Y(h, z)a, b)  \\
&=& Res_z(a, Y(e^{zL(1)}(-z^{-2})^{L(0)}h, z^{-1})b)\ \ (\mbox{by (\ref{invbilform1})})\\
&=& -Res_z z^{-2}(a, Y(h+zL(1)h, z^{-1})b) \\
&=&-(a, h(0)b)-(a, (L(1)h)(-1)b) \\
&=&-\mu(a, b)-(a, L'(1)h-2h(1)h)(-1)b) \\
&=&-\mu(a, b)+2(a, h(1)h)(-1)b) \\
&=&-\mu(a, b)+\nu(a, b).
\end{eqnarray*}
Here, we used the assumption that $W$ is self-dual and of CFT-type to conclude
that $L'(1)W_1=0$, and in particular $L'(1)h=0$. Thus we have obtained
\begin{eqnarray}\label{ip}
(\lambda+\mu-\nu)(a, b)=0.
\end{eqnarray}

If $\lambda+\mu\not= \nu$ then we must have $(a, b)=0$ for all choices of
$a, b$, and this is exactly what (v) says. Because the bilinear form $( \ , \ )$ is nondegenerate, it 
follows that it must induce a perfect pairing between $A^{\lambda}$ and $A^{\mu}$
whenever $\lambda+\mu=\nu$.\ So (vi) holds.

\medskip
Finally, taking $\lambda = 0$, we know that $A^0 = \mathbb{C}\mathbf{1}$ by (ii).
Thus $A^0$ pairs with $A^{\nu}$ and $\dim A^{\nu}=1$. Because
$(\mathbf{1}, t)\not= 0$ by (\ref{1tperp}), and $T$ is an eigenspace for $h(0)$
(Lemma \ref{lemmaM}), we see that $A^{\nu}=T$. This proves (iii), and completes the proof of
the Theorem. $\hfill \Box$

\section{The bilinear form in the shifted case}
We return to the issue, introduced in Section 5, of the nature of the radical 
of the bilinear form $\langle \ , \ \rangle$ for an $\mathbb{N}$-graded vertex operator algebra $V$, 
assuming now that $V$ is a shift of a simple, self-dual vertex operator algebra $W=(W, Y, \mathbf{1}, \omega')$ of CFT-type as in Section \ref{Sshift}.\ We will also assume that $\dim V_0\geq 2$.

\medskip
We continue to use the notations of Section $3-5$.\ We shall
see that the question raised in Remark \ref{remQ} has an affirmative answer in this case,
and that $rad\langle \ , \ \rangle/N$ is \emph{exactly} the nilpotent radical $N_1/N$ of the Lie algebra
$V_1/N$ when $W_1$ is reductive.\ The precise result is as follows.

\begin{theorem}\label{thmshift} We have
\begin{eqnarray}\label{inters}
 \oplus_{m\geq 2} W_{m, m-1} \subseteq Ann_{V_1}(t(-1)) = rad\langle \ , \ \rangle.
\end{eqnarray}
Moreover, if $W_1$ is \emph{reductive} then
\begin{eqnarray}\label{eqs}
 N_1=\oplus_{m\geq 2} W_{m, m-1} = Ann_{V_1}(t(-1))= rad\langle \ , \ \rangle.
\end{eqnarray}
\end{theorem}

 \medskip
 Recall that 
$V_1$ is a Leibniz algebra, $N= L(-1)V_0$ is a $2$-sided ideal in $V_1$, 
$V_1/N$ is a Lie algebra, and $N_1/N$ is the \emph{nilpotent radical} of $V_1/N$
(\ref{Leibnilraddef}).\ Because $W$ is a VOA
of CFT-type then $W_1$ is a Lie algebra with bracket $a(0)b\ (a, b \in W_1)$, and 
 $W_{1, 0} =C_{W_1}(h)$ is the \emph{centralizer} of $h$ in $W_1$.\

\begin{lemma}\label{lemmahPM} $h\in P$.
\end{lemma}
\noindent 
{\bf Proof}.\ We have to show that $\langle h, M\rangle = (h(1)M, t)=0$.\ Since $M$ is an ideal in $V_1$ then $M$ is the direct sum of its $h(0)$-eigenspaces.\
Let $M^p = \{m\in M\ | \ h(0)m=pm\}$.\
Now $h(0)h(1)m=h(1)h(0)m=ph(1)m\ (m\in M^p)$, showing that
$h(1)M^p\subseteq A^p$.\ If $p\not= 0$ then $(A^p, t)=0$ by Theorem \ref{derham},
so that $(h(1)M^p, t)=0$ in this case.

\medskip
It remains to establish that $(h(1)M^0, t)=0$.\ To see this, first note that because $W$ is self-dual then
$L'(1)W_1=0$.\ Since $L'(1)=L(1)+2h(1)$ then $h(1)W_1 = L(1)W_1$, and in particular
$h(1)M^0=L(1)M^0$ (because $M^0\subseteq V_1^0=W_{1, 0}\subseteq W_1$).\ Therefore,
$(h(1)M^0, t)=(L(1)M^0, t)=(M^0, L(-1)t)=0$,
where the last equality holds by Lemma \ref{lemmaM}.\ 
The Lemma is proved. $\hfill \Box$

\begin{lemma}\label{lemmahnotann} $h\notin Ann_{V_1}(t(-1))\cup rad\langle \ , \ \rangle$.
\end{lemma}
\noindent
{\bf Proof}. First recall that $h(1)h=\nu/2\mathbf{1}$.\ Then we have
\begin{eqnarray*}
\langle h, h\rangle &=& (h(1)h, t) = \nu/2(\mathbf{1}, t)\not= 0.
\end{eqnarray*}
Here, 
$\nu\not= 0$ thanks to Theorem \ref{derham} and because we are assuming that
$\dim V_0\geq 2$.\ Because $h\in V_1$, this shows that $\langle h, V_1\rangle \not= 0$, so that $h\notin rad\langle \ , \ \rangle$.

\medskip
Next, using (\ref{Lshift}) we have $L(1)h=L'(1)h-2h(1)h$.\ Because $W$ is assumed to be self-dual then
$L'(1)h=0$, so that $L(1)h=-2h(1)h =-\nu\mathbf{1}$.\ Now
\begin{eqnarray*}
(t(-1)h, h)= (h(-1)t-L(-1)h(0)t, h).
\end{eqnarray*}
Therefore,
\begin{eqnarray*}
(L(-1)h(0)t, h) &=& (h(0)t, L(1)h)= -\nu^2(t, \mathbf{1}).
\end{eqnarray*}
Also,
\begin{eqnarray*}
(h(-1)t, h) &=& (t, -h(1)h-(L(1)h)(0)h) = -\nu/2(t, \mathbf{1}).
\end{eqnarray*}
Therefore,
\begin{eqnarray*}
(t(-1)h, h) &=&(\nu^2-\nu/2)(t, \mathbf{1}).
\end{eqnarray*}

\medskip
Because $\nu$ is a positive integer, the last displayed expression is nonzero.\
Therefore, $t(-1)h\not= 0$, i.e., $h\notin Ann_{V_1}(t(-1))$.\ This completes the proof of the
Lemma. $\hfill \Box$

\medskip
\noindent
We turn to the proof of Theorem \ref{thmshift}.\ First note that by combining Lemmas
\ref{lemmahPM} and \ref{lemmahnotann} together with Proposition \ref{relari}, we see
that cases (ii) and (iii) of Proposition \ref{relari} \emph{cannot} hold.\
Therefore, case (i) must hold, that is
\begin{eqnarray*}
rad\langle \ , \ \rangle = Ann_{V_1}(t(-1)).
\end{eqnarray*}

From (\ref{h0decomp0}) it is clear that, up to scalars, $\mathbf{1}$ is the only
state in $V_0$ annihilated by $h(0)$.\ It then follows from Lemma \ref{lemmaM} that
$J = \oplus_{m\geq 1} W_{m, m}$.\ In particular, $(W_{m, m}, t)=0\ (m\geq 1)$ by
(\ref{Tchar}).

\medskip
Now let $u\in W_{m, m-1}, v\in W_{k, k-1}$ with $m\geq 1, k\geq 2$.\ Then
$u(1)v \in V_0$ and $L'(0)u(1)v = (m+k-2)u(1)v$.\ Therefore, $u(1)v \in W_{m+k-2, m+k-2}$,
and because $m+k-2\geq 1$ it follows that
\begin{eqnarray*}
\langle u, v\rangle = (u(1)v, t) = 0.
\end{eqnarray*}
Because this holds for all $u\in W_{m, m-1}$ and all $m\geq 1$, we conclude that
$v\in$ rad$\langle \ , \rangle$.\ This proves that $\oplus_{m\geq 2}W_{m, m-1} \subseteq$ rad$\langle \ , \ \rangle$.\ Now (\ref{inters}) follows immediately.

\medskip
Now suppose that $W_1$ is reductive.\ 
Because $W$ is self-dual and of CFT-type, it has (up to scalars) a unique nonzero invariant bilinear form.\ Let us denote it by $(( \ , \ ))$. In particular, we have 
\begin{eqnarray*}
((u, v))\mathbf{1} = u(1)v\ \ (u, v \in W_1).
\end{eqnarray*}

\medskip
Because $V$ is simple and $V, W$ have the same set of fields,
then $W$ is also simple.\ In particular, $(( \ , \ ))$ must be nondegenerate.
Now if $L$ is a (finite-dimensional, complex) reductive Lie algebra equipped
with a nondegenerate, symmetric invariant bilinear form, then the restriction of the form
to the centralizer of any semisimple element in $L$ is also nondegenerate.\ In the present situation, this tells us that
the restriction of $(( \ , \ ))$ to $C_{W_1}(h)$ is nondegenerate.

\medskip
On the other hand, we have 
\begin{eqnarray*}
\langle u, v \rangle = (u(1)v, t) = ((u, v))(\mathbf{1}, t)
\end{eqnarray*}
and $(\mathbf{1}, t)\not= 0$ by (\ref{1tperp}).\ This shows that the restrictions of $\langle \ , \ \rangle$ and 
$(( \ , \ ))$ to $C_{W_1}(h)$ are \emph{equivalent} bilinear forms.\ Since the latter  is nondegenerate,
so is the former.\ Therefore, rad$\langle \ , \ \rangle \cap C_{W_1}(h)=0$.\
Now the second and third equalities of (\ref{eqs}) follows from (\ref{inters}) and the decomposition
$V_1 = C_{W_1}(h)\oplus \oplus_{m\geq 2} W_{m, m-1}$.

\medskip
To complete the proof of the Theorem it suffices to prove the next result.
\begin{lemma}\label{prop4.1} We have
\begin{eqnarray}\label{N0decomp1}
N_1 = Nilp(C_{W_1}(h)) \oplus_{m\geq 2}W_{m, m-1}.
\end{eqnarray}
In particular, if $W_1$ is a \emph{reductive} Lie algebra then
\begin{eqnarray}\label{N0decomp2}
N_1 = \oplus_{m\geq 2} W_{m, m-1}.
\end{eqnarray}
\end{lemma}
\noindent
{\bf Proof.}\
Let $u \in W_{k, k-1}, v \in W_{m, m-1}$ with $k\geq 1, m \geq 2$.\ Then
$v(0)u \in W_{m+k-1}\cap V_1\subseteq W_{m+k-1, m+k-2}$ and $m+k-1>k$.\
This shows that $\oplus_{m\geq 2}W_{m, m-1}$ is an ideal in $V_1$.\ Moreover, because there is a maximum integer $r$ for which $W_{r, r-1}\not= 0$, it follows that $v(0)^l V_1=0$ for large enough $l$, so that
the left adjoint action of $v(0)$ on $V_1$ is \emph{nilpotent}.\ This shows that
$\oplus_{m\geq 2}W_{m, m-1}$ is a nilpotent ideal.\ Because $h(0)$ acts
on $W_{m, m-1}$ as multiplication by $m-1$, then $W_{m, m-1}=[h, W_{m, m-1}]$ for $m\geq 2$, whence
in fact
$\oplus_{m\geq 2}W_{m, m-1} \subseteq N_1$.\ Then (\ref{N0decomp1}) follows immediately.

\medskip
Finally, if $W_1$ is reductive, the centralizer of any semisimple element in $W_1$
is also reductive.\ In particular, this applies to $C_{W_1}(h)$ since $h(0)$ is indeed semisimple, so (\ref{N0decomp2}) follows from (\ref{N0decomp1}).\ This completes the proof of
the Lemma, and hence also that of Theorem \ref{thmshift}. $\hfill \Box$

\bigskip
The VOA $W$ is \emph{strongly regular} if it is self-dual and CFT-type as well
as both \emph{$C_2$-cofinite} and \emph{rational}.\ For a general discussion of such VOAs
see, for example, \cite{M}.\ It is known (\cite{M} and \cite{DM1}, Theorem 1.1) that
in this case, $W_1$ is necessarily reductive.\ Consequently, we deduce from
Theorem \ref{thmshift} that the following holds.
\begin{corollary}\label{cor4.3} Suppose that $W$ is a strongly regular VOA, and that $V$ is a self-dual,
$\mathbb{N}$-graded
VOA obtained as a shift of $W$.\ Then $rad\langle \ , \ \rangle = Nilp(V_1)$. $\hfill \Box$
\end{corollary}

\medskip
\noindent
\begin{remark} The Corollary applies, for example, to the shifted theories $V=L_{\hat{sl_2}}(k, 0)^H$
discussed in Section \ref{SSsl2} below.\ In this case, one can directly compute the relevant quantities.
\end{remark}

\section{The $C_2$-cofinite case}
In this Subsection we are mainly concerned with simple VOAs $V$ that are
self-dual and $\mathbb{N}$-graded as before, but that are also \emph{rational}, or
\emph{$C_2$-cofinite}, or both. 

\medskip
To motivate the main results of the present Subsection, we recall some results about vertex operator algebras $V$ with $V_0=\mathbb{C}\mathbf{1}$.\ In this case, $V_1$ is a Lie algebra, and if
$V$ is strongly regular then
$V_1$ is \emph{reductive} (\cite{DM1}, Theorem 1.1).\ It is also known (\cite{DM4},
Theorem 3.1) that if $V$ is $C_2$-cofinite, but not necessarily rational, and
$S\subseteq V_1$ is a Levi factor, then the vertex operator subalgebra $U$ of $V$
generated by $S$ satisfies
\begin{eqnarray}\label{ssLAgen}
U\cong L_{\hat{\frak{g}}_1}(k_1,0)\oplus...\oplus L_{\hat{\frak{g}}_r}(k_r,0),
\end{eqnarray}
i.e., a direct sum of simple affine Kac-Moody Lie algebras $L_{\hat{\frak{g}}_j}(k_j,0)$ of positive integral level $k_j$. 

\medskip
We want to know to what extent these results generalize to the more general case when $dimV_0>1$.\
With $N=L(-1)V_0$ as before, we have seen  that $V_1/N$ is a Lie algebra.\ Now $V_0=\mathbb{C}\mathbf{1}$ precisely when $N=0$, but the natural guess
that $V_1/N$ is reductive if $V$ is rational and $C_2$-cofinite is generally false.\ Thus we need 
to understand the nilpotent radical $N_1/N$ of this Lie algebra. That is where the  bilinear form
$\langle \ , \ \rangle$ comes in.\ These questions
are naturally related to the issue, already addressed in Section \ref{SS2}, of the structure of the subalgebra of $V$ generated by a Levi subalgebra of $V_1$.\ The main result is

 \begin{theorem}\label{thmC2} Let $V$ be a simple, self-dual, $\mathbb{N}$-graded vertex operator algebra that is 
 $C_2$-cofinite, and let $V_1=B\oplus S$ with Levi factor
 $S$ and solvable radical $B$. Then the following hold.
\begin{enumerate}
\item $N_1\subseteq rad\langle \ , \ \rangle\subseteq B$, and the restriction of $\langle\ , \ \rangle$ to $S$ is non-degenerate.\ In particular, $a(1)b\in\mathbb{C}{\bf 1}$ for all $a,b\in S$.
\item If $U$ is the vertex operator algebra generated by $S$ then 
$U$ satisfies (\ref{ssLAgen}).
\end{enumerate}
\end{theorem}

We start with
\begin{proposition}\label{propgen} Let $V=\oplus_{n=0}^{\infty}V_n$ be an $\mathbb{N}$-graded  vertex operator algebra such that $dim~V_0>1$. Let $X=\{x^i\}_{i\in I}\cup \{y^j\}_{j\in J}$ be a set of homogeneous elements in $V$ which are representatives of a basis of $V/C_2(V)$. Here, $x^i$ are vectors whose weights are greater than or equal to 1 and $y^j$ are vectors whose weights are zero. Then $V$ is spanned by elements of the form 
$$x^{i_1}(-n_1)...x^{i_s}(-n_s)y^{j_1}(-m_1)...y^{j_k}(-m_k){\bf 1}$$
where $n_1>n_2>....>n_s>0$ and $m_1\geq m_2\geq...\geq m_k>0$.
\end{proposition}
\noindent
{\bf Proof}.\  The result follows by modifying the proof of Proposition 8 in \cite{GN}.
$\hfill \Box$

\medskip
Notice that for a Lie algebra $W\subset V_1$, we have $u(0)v=-v(0)u$ for $u,v\in W$. Hence, $L(-1)u(1)v=0$ and $u(1)v\in \mathbb{C}{\bf 1}$. Moreover, we have $\langle \ , \ \rangle{\bf 1}=(u(1)v,t){\bf 1}= u(1)v$ for $u,v\in W$.

\medskip
Let $\frak{g}$  be a finite dimensional simple Lie algebra, $\mathfrak{h}\subset\frak{g}$ be a Cartan sub-algebra, and $\Psi$ be the associated root system with simple roots $\Delta$. Also, we set $(\cdot,\cdot)$ be the non-degenerate symmetric invariant bilinear form on $\frak{g}$ normalized so that the longest positive root $\theta\in \Psi$ satisfies $(\theta,\theta)=2$.\ The corresponding affine Kac-Moody Lie algebra $\hat{\frak{g}}$ is defined as
$$\hat{\frak{g}}=\frak{g}\otimes\mathbb{C}[t,t^{-1}]\oplus\mathbb{C} K,$$
where $K$ is central and the bracket  is defined  for $u,v\in \frak{g}$, $m,n\in \mathbb{Z}$ as
$$[u(m),v(n)]=[u,v](m+n)+m\delta_{m+n,0}(u,v)K\ \  (u(m)=u\otimes t^m).$$

Let $W\subset V_1$ be a Lie algebra such that $\frak{g}\subset W$ and $\langle \ , \ \rangle$ is non-degenerate on $W$. If $\langle \ , \ \rangle$ is non-degenerate on $\frak{g}$, then the map 
$$\hat{\frak{g}}\rightarrow End(V); u(m)\mapsto u(m),~u\in \frak{g}, m\in\mathbb{Z},$$ 
is a representation of the affine Kac-Moody algebra $\hat{g}$ of level $k$ where 
$$\langle \ , \ \rangle=u(1)v=k(u,v)\text{ for }u,v\in \frak{g}.$$

\begin{theorem}\label{wlko} Let $W\subset V_1$ be a Lie algebra such that $\langle \ , \ \rangle$ is nondegenerate on $W$. If $\frak{g}$ is a simple Lie-subalgebra of $W$ and $U$ is the vertex operator sub-algebra of $V$ generated by $\frak{g}$, then $\langle \ , \ \rangle$ is nondegenerate on $\frak{g}$, and $U$ is isomorphic to 
$L_{\hat{\frak{g}}}(k,0)$. Furthermore, $k$ is a positive integer and $V$ is an integrable $\hat{\frak{g}}$-module.
\end{theorem}
\noindent
{\bf Proof}.\  We will follow the proof in Theorem 3.1 of \cite{DM4}.\ First, assume that $\frak{g} =sl_2(\mathbb{C})$ with standard basis $\alpha,x_{\alpha},x_{-\alpha}$.\ Hence, $(\alpha,\alpha)=2$.\ Since each homogeneous subspace of $V$ is a completely reducible $\frak{g}$-module, then $V$ is also a completely reducible 
$\frak{g}$-module.\ A non-zero element $v\in V$ is called a weight vector for $\frak{g}$ of $\frak{g}$-weight $\lambda$ ($\lambda\in\mathbb{C}\alpha$) if $\alpha (0)v=(\alpha,\lambda)v$. Here, 
$\lambda\in \frac{1}{2}\mathbb{Z}\alpha$.

\medskip
We now make use of Proposition \ref{propgen}.\ Let $X=\{x^i\}_{i\in I}\cup \{y^j\}_{j\in J}$ be a set of homogeneous weight vectors in $V$ which are representatives of a basis of $V/C_2(V)$.\ The $x^i$ are vectors whose weights are greater than or equal to 1 and $y^j$ are vectors whose weights are zero. Since $X$ is finite, there is a nonnegative elements $\lambda_0=m\alpha\in\frac{1}{2}\mathbb{Z}\alpha$ such that the weight of each element in $X$ is bounded above by $\lambda_0$. For any integer $t\geq 0$, we have
\begin{eqnarray}\label{nbasis}
\oplus_{n\leq t(t+1)/2}V_n&\subseteq &Span_{\mathbb{C}}\{x^{i_1}(-n_1)...x^{i_s}(-n_s)y^{j_1}(-m_1)...y^{j_r}(-m_{r}){\bf 1}~|~\nonumber\\
& &\ \ \ \ \ \ \ \ \ \ \ \ n_1>n_2>...>n_s>0, m_1\geq m_2\geq ...\geq m_r>0,\nonumber\\
& &\ \ \ \ \ \ \ \ \ \ \ \ 0\leq s,r\leq t\}.
\end{eqnarray}
Furthermore, if $n\leq \frac{t(t+1)}{2}$, then a $\frak{g}$-weight vector in $V_n$ has $\frak{g}$-weight less than or equal to $2t\lambda_0=2tm\alpha$.

\medskip
Let $l$ be an integer such that $l+1>4m$ and we let $$u=(x_{\alpha})(-1)^{l(l+1)/2}{\bf 1}.$$ We claim that $u=0$. Assume $u\neq 0$. By (\ref{nbasis}), we can conclude that the $\frak{g}$-weight of $u$ is at most $2lm\alpha$. This contradicts with the direct calculation which shows that the $\frak{g}$-weight of $u$ is $\frac{l(l+1)}{2}\alpha$. Hence, $u=0$. This implies that $U$ is integrable. Furthermore, we have $V$ is integrable, $k$ is a positive integer and $\langle \ , \ \rangle$ is non-degenerate.

\medskip
 This proves the Theorem for $\frak{g}=sl_2$.\ The general case follows easily from this (cf.\ \cite{DM4}).\ $\hfill \Box$

\begin{lemma}\label{snonde} Let $S$ be a Levi sub-algebra of $V_1$. Then $\langle \ , \ \rangle$ is non-degenerate on $S$ and $Rad\langle \ , \ \rangle\cap S=\{0\}$.
\end{lemma}
\begin{proof} 
Clearly, for $u,v\in S$, we have $u(1)v\in\mathbb{C}{\bf 1}$. Let $f:S\times S\rightarrow\mathbb{C}{\bf 1}$ be a map defined by $f((u,v))=u(1)v$. Since $u(1)v=v(1)u$ and $$(w(0)u)(1)v=-(u(0)w)(1)v=-(u(0)w(1)v-w(1)u(0)v)=w(1)u(0)v$$ for $u,v,w\in S$, we can conclude that $f$ is a symmetric invariant bilinear form on $S$. For convenience, we set $X=Rad(f)$. Since $S$ is semi-simple and $X$ is a $S$-module, these imply that $S=X\oplus W$ for some $S$-module $W$. Note that $W$ and $X$ are semi-simple and $S\cap Rad\langle \ , \ \rangle\subseteq X$. 

\medskip
For $u,v\in X$, we have $u(1)v=0$. Hence, the vertex operators $Y(u,z)$, $u\in X$, generate representation of the loop algebra in the sense that
$$[u(m),v(n)]=(u(0)v)(m+n),\text{ for }u,v\in X.$$ 
Following the proof of Theorem 3.1 in \cite{DM4}, we can show that the representation on $V$ is integrable and the vertex operator sub-algebra $U$ generated by a simple component of $X$ is the corresponding simple vertex operator algebra $L(k,0)$ and $k=0$. However, the maximal submodule of the Verma module $V(0,0)$, whose quotient is $L(0,0)$, has co-dimension one. This is not possible if $X\neq \{0\}$. Consequently, we have $X=0$ and $S\cap Rad\langle \ , \ \rangle=\{0\}$. Hence $\langle \ , \ \rangle$ is non-degenerate on $S$. 
\end{proof}

Theorem \ref{thmC2} follows from these results.

\section{Examples of shifted vertex operator algebras} \label{SSsl2}
To illustrate previous results, in this Section we consider some particular classes of shifted vertex operator algebras.

\subsection{Shifted $\widehat{sl_2}$}
 We will show that the 
simple vertex 
operator algebra (WZW model) $L_{\widehat{sl_2}}(k, 0)$ corresponding to affine 
$sl_2$ at (positive integral) level $k$ has a canonical shift to an $\mathbb{N}$-graded
vertex operator algebra $L_{\hat{sl_2}}(k, 0)^H$, and that the resulting de Rham structure on $V_0$
is that of complex projective space  $\mathbb{CP}^k$.\ The precise result is the following.

\begin{theorem} Let $e, f, h$ be Chevalley generators of $sl_2$, and set $H=h/2$.\ Then
the following hold:
\begin{eqnarray*}
&&(a)\ L_{\hat{sl_2}}(k,0)^H \ \mbox{is a simple,  $\mathbb{N}$-graded, self-dual vertex operator algebra}. \\
&&(b)\ \mbox{The algebra structure on the zero weight space of
$L_{\hat{sl_2}}(k,0)^H$ is} \\
&&\ \ \ \ \ \mbox{isomorphic to $\mathbb{C}[x]/\langle x^{k+1}\rangle$, where $x=e$}.
\end{eqnarray*}
 \end{theorem}
\noindent
{\bf Proof}.\ Let $W=L_{\hat{sl_2}}(k, 0)$.\ It is spanned by states $v_{IJK}:=e_{I}f_Jh_K\mathbf{1}$, where 
we write $e_I= e(-l_1)\hdots e(-l_r),  f_J = f(-m_1) \hdots f(-m_s), h_K=h(-n_1)\hdots h(-n_t)$
for $l_i, m_i, n_i>0.$\ Note that $v_{IJK}\in W_n$, where $n=\sum l_i +\sum m_i+\sum n_i$

\medskip
Recall from (\ref{Lshift}) that $L_H(0)=L(0)-H(0)$.\ We have
\begin{eqnarray*}
[H(0),e(n)]&=& [H, e](n)=e(n),\\
{[H(0),f(n)]}&=&[H, f](n)=-f(n),\\
{[H(0),h(n)]}&=&[H, h](n)=0.
\end{eqnarray*}
Then 
$H(0)v_{IJK} = (r-s)v_{IJK}$, so that
\begin{eqnarray}\label{IJKwt}
L_H(0)v_{IJK} = \left( \sum (l_i-1) +\sum (m_j+1)+\sum n_k\right)v_{IJK}
\end{eqnarray}

It is well-known
(e.g., \cite{DL}, Propositions 13.16 and 13.17) that $Y(e, z)^{k+1}=0$.\
Thus we may take $r$ to be no greater than $k$.\ It follows from
(\ref{IJKwt}) that the eigenvalues of $L_H(0)$ are integral and bounded below by $0$, 
and that the eigenspaces 
are finite-dimensional.\ Therefore, $V:=W^H$ is indeed
an $\mathbb{N}$-graded vertex operator algebra.\ Because $W$ is simple then so too is $V$, since they share the same fields.

\bigskip
Next we  show that $V$ is self-dual, which amounts to the assertion
that $L_H(1)V_1$ is \emph{properly} contained in $V_0$.\ 
Observe from
(\ref{IJKwt}) that the states $e(-1)^{p}{\bf 1}$
$(0 \leq p \leq k)$ span in $V_0$, while
$V_1$ is spanned by the  states
$\{h(-1)e(-1)^i{\bf 1},e(-2)e(-1)^i{\bf 1}~|~ 0\leq i\leq k-1\}$.\
We will show that $e(-1)^k\mathbf{1}$ does \emph{not} lie in the image
of $L_H(1)$.

\medskip
For $g\in sl_2, m\geq 1$, we have $$[L(1),g(-m)]=mg(1-m).$$ Since $L(1)e=0$ and $$L(1)e(-1)^{j+1}{\bf 1}=e(-1)L(1)e^{j}(-1){\bf 1}+e(0)e^{j}(-1){\bf 1}=e(-1)L(1)e^{j}(-1){\bf 1}$$ for $j\geq 0$, we can conclude by induction that 
\begin{equation*}L(1)e(-1)^i{\bf 1}=0\text{ for all }i\geq 1.\end{equation*} 
Similarly, because $H(1)e=0$ and
\begin{eqnarray*}
H(1)e(-1)^{j+1}{\bf 1}=e(-1)H(1)e(-1)^j{\bf 1}+e(0)e(-1)^{j}{\bf 1}=e(-1)H(1)e(-1)^j{\bf 1},
\end{eqnarray*} 
then
\begin{equation*}H(1)e(-1)^i{\bf 1}=0\text{ for all }i\geq 0.\end{equation*}

\medskip
We can conclude that for $0\leq i\leq k-1$,
\begin{eqnarray*}
L_H(1)h(-1)e(-1)^{i}{\bf 1}&=&(L(1)-2H(1))h(-1)e(-1)^{i}{\bf 1}\\
&=&2ie(-1)^i{\bf 1}-2ke(-1)^i{\bf 1}\\
&=&2(i-k)e(-1)^i{\bf 1},
\end{eqnarray*}
while
\begin{eqnarray*}
 L_H(1)e(-2)e(-1)^j{\bf 1}=(L(1)-2H(1))e(-2)e(-1)^j{\bf 1}=0.
 \end{eqnarray*}
Our assertion that $e(-1)^k\mathbf{1}\notin imL_H(1)$ follows from these calculations.\
This establishes part (a) of the Theorem.

\medskip
Finally, if we set $x:= e(-1)\mathbf{1}=e$ then by induction
$e(-1)^i\mathbf{1}=x.x^{i-1}=x^i$, so the algebra structure on $V_0$ is
isomorphic $\mathbb{C}[x]/x^{k+1}$ and part (b) holds.\ This completes the proof of the Theorem.
$\hfill \Box$

\noindent
\begin{remark} Suitably normalized, the invariant bilinear form on 
$V_0$ satisfies
$(x^p, x^q) = \delta_{p+q, k}$ (cf.\ Theorem \ref{derham}).\ $V_0$ can be identified with
the de Rham cohomology of $\mathbb{CP}^k$ ($x$ has degree $2$)
equipped with the pairing arising from Poincar\'{e} duality. 
\end{remark}

\subsection{Shifted lattice theories}
Let $L$ be a positive-definite even lattice of rank $d$ with inner product 
$( \ , \ ):L\times L\rightarrow \mathbb{Z}$.\ Let $H=\mathbb{C}\otimes L$ be the corresponding complex linear space equipped with the 
$\mathbb{C}$-linear extension of $( \ , \ )$.\ The dual lattice of $L$ is 
$$L^{\circ}=\{~f \in \mathbb{R}\otimes L~|~(f,\alpha)\in\mathbb{Z}\text{ all }\alpha\in L\}.$$
Let $(M(1),Y,{\bf 1}, \omega_L)$ be the free bosonic vertex operator algebra based on $H$ and let $(V_L, Y,{\bf 1},\omega_L)$ be the corresponding lattice vertex operator algebra.\ Both vertex operator algebras have central charge $d$, and the Fock space of $V_L$ is 
$$V_L=M(1)\otimes \mathbb{C}[L],$$ where $\mathbb{C}[L]$ is the group algebra of $L$. 

\medskip
For a state $h\in H\subset (V_L)_1$, we set $\omega_h=\omega_L+h(-2){\bf 1}$,
with $V_{L,h}=(V_L,Y,{\bf 1}, \omega_h)$.
\begin{lemma}(\cite{DM3}).\ Suppose that $h\in L^0$.\ Then 
$V_{L,h}$ is a vertex operator algebra, and it is self-dual if, and only if, $2h\in L$.
\end{lemma} 

\medskip
For the rest of this section, we assume that $0\not=h\in L^{\circ}$ and $2h\in L$,
so that $V_{L,h}$ is a self-dual, simple vertex operator algebra.\ Set 
$$Y(\omega_h,z)=\sum_{n\in\mathbb{Z}}L_h(n)z^{-n-2}.$$\ Then
\begin{eqnarray*}
L_h(0)(u\otimes e^{\alpha})&=&(n+\frac{1}{2}(\alpha,\alpha)-(h,\alpha))u\otimes e^{\alpha}\ (u\in M(1)_n).
\end{eqnarray*}
It follows that $V_{L, h}$ is $\mathbb{N}$-graded if, and only if,
the following condition holds:
\begin{eqnarray}\label{alphaineq}
(2h, \alpha) \leq (\alpha, \alpha)\ \ (\alpha\in L).
\end{eqnarray}

\medskip
From now on, we \emph{assume that (\ref{alphaineq}) is satisfied}.\ It is equivalent to the condition
$(\alpha-h, \alpha-h)\geq (-h, -h)$, i.e., $-h$ has the least (squared) length
among all elements in the coset $L-h$.\ Set
\begin{eqnarray*}
 A:=\{\alpha\in L\ | \ (\alpha, \alpha) = (2h, \alpha)\}.
 \end{eqnarray*}
 Note that $0, 2h\in A$.\ We have
\begin{eqnarray*}
(V_{L,h})_0&=&Span_{\mathbb{C}}\{e^{\alpha}~|~\alpha\in A\}.
\end{eqnarray*}
We want to understand the commutative algebra structure of $(V_{L, h})_1$, defined by
the $-1^{th}$ product $a(-1)b$.\ The identity element is $\mathbf{1}=e^{0}$.

\medskip
First note that if $\alpha, \beta\in A$ then
$(-h, -h)\leq (\alpha+\beta-h, \alpha+\beta-h)=(h, h)+2(\alpha, \beta)$ shows that 
$(\alpha, \beta)\geq 0$.\ Moreover, $(\alpha, \beta)=0$ if, and only if, $\alpha+\beta \in A$.
We employ standard notation for vertex operators in the lattice theory $V_L$ (\cite{LL}).
Then
\begin{eqnarray*}
e^{\alpha}(-1)e^{\beta}&=& Res_{z}z^{-1}E^-(-\alpha,z)E^+(-\alpha,z)e_{\alpha}z^{\alpha}\cdot e^{\beta}\\
&=&\epsilon(\alpha, \beta) Res_{z}z^{(\alpha,\beta)-1}E^-(-\alpha,z)E^+(-\alpha,z)e^{\alpha+\beta},
\end{eqnarray*}
where 
\begin{eqnarray*}
E^-(-\alpha, z)E^+(-\alpha, z)&=&exp\left\{-\sum_{n> 0}\frac{\alpha(-n)}{n}z^{n}\right\}
exp\left\{\sum_{n> 0}\frac{\alpha(n)}{n}z^{-n}\right\}.
\end{eqnarray*}
It follows that
\begin{eqnarray}\label{product}
e^{\alpha}(-1)e^{\beta} = \left\{ \begin{array}{cc}
\epsilon(\alpha, \beta)e^{\alpha+\beta}& \mbox{if $\alpha+\beta\in A$} \\
0 & \mbox{otherwise}
\end{array} \right.
\end{eqnarray}

\medskip
If $0\not=\alpha\in A$ then $(2h, \alpha)=(\alpha, \alpha)\not= 0.$ Thus $2h+\alpha \notin A$, and the last calculation shows that $e^{\alpha}(-1)e^{2h}=0$.\ It follows that $e^{2h}$ spans
the unique minimal ideal $T \subseteq (V_{L, h})_1$

\medskip
Recall (\cite{LL}) that $\epsilon: L\times L\rightarrow \left\{\pm 1\right\}$ is a 
(bilinear) $2$-cocycle satisfying $\epsilon(\alpha, \beta)\epsilon(\beta, \alpha)=(-1)^{(\alpha, \beta)}$.\
Thus we have proved
\begin{lemma} There are signs $\epsilon(\alpha, \beta)=\pm 1$ such that multiplication in
$(V_{L, h})_1$ is given by (\ref{product}). The minimal ideal $T$ is spanned by $e^{2h}$. $\hfill \Box$
\end{lemma}

\end{document}